\documentclass[12pt]{amsart}
\usepackage{amsfonts}

\newtheorem*{theorem}{Theorem}
\newtheorem*{conjecture}{Conjecture}

\newcommand{\Tr}{{\rm tr}\,}

\newcommand{\bR}{\mathbb{R}}

\newcommand{\bN}{\mathbb{N}}
\newcommand{\cD}{\mathcal{D}}

\begin{document}

\baselineskip 7mm

\title{On the $S$-matrix conjecture 
\footnote{The paper will appear in Linear Algebra and its Applications.}}

\author{Roman Drnov\v sek}

\date{\today}

\begin{abstract}
\baselineskip 6.5mm
Motivated with a problem in spectroscopy,  Sloane and Harwit conjectured in 1976 what is the minimal Frobenius norm of the inverse of a matrix
having all entries from the interval $[0, 1]$. In 1987, Cheng proved their conjecture in the case of odd dimensions, while for even dimensions 
he obtained a slightly weaker lower bound for the norm. His proof is based on the Kiefer-Wolfowitz equivalence theorem 
from the approximate theory of optimal design.  In this note we give a short and simple proof of his result.
\end {abstract}

\maketitle
\baselineskip 6.7mm
\noindent
{\it Math. Subj.  Classification (2010)}: 15A45, 15A60 \\

A {\it Hadamard matrix} is a square matrix with entries in $ \{-1, 1\}$ whose rows and hence columns are mutually orthogonal. 
In other words, a Hadamard matrix of order $n$ is a $\{-1, 1\}$-matrix $A$ satisfying $A A^T = n I$, i.e.,  
$\frac{1}{\sqrt{n}} A$ is a unitary matrix. 

An {\it S-matrix} of order $n$ is a $\{0, 1\}$-matrix formed by taking a Hadamard matrix of order $n+1$ in which the entries 
in the first row and column are $1$, changing $1$'s to $0$'s and $-1$'s to $1$'s, and deleting the first row and column. 

The {\it Frobenius norm} of a real matrix $A = [a_{i, j}]_{i, j =1}^n$ is defined as 
$$ \|A\|_F =  \left( \sum_{i=1}^n \sum_{j=1}^n a_{i, j}^2 \right)^{1/2} \ . $$
It is associated to the inner product defined by 
$$ \langle A, B \rangle = \Tr ( A B^T) = \sum_{i=1}^n \sum_{j=1}^n a_{i, j} b_{i, j} . $$
Let $\cD_n$ denote the set of all matrices $A = [a_{i, j}]_{i, j =1}^n$  whose entries are in the interval $[0, 1]$.

In 1976, Sloane and Harwit \cite{SH} posed the following conjecture. See also \cite[p.59]{HS} or 
\cite[Conjecture 11]{Zh}.

\begin{conjecture}
If $A \in \cD_n$ is a nonsingular matrix, then 
$$  \| A^{-1}\|_F \ge \frac{2 n}{n+1} \ ,  $$
where the equality holds if and only if  $A$ is an S-matrix.
\end{conjecture}

This conjecture arose from a problem in spectroscopy.  
A detailed discussion of its applications in spectroscopy can be found in \cite{HS}. 
The conjecture has been proved in recent papers \cite{ZJH}, \cite{Zou} and  \cite{Hu} for some special matrices.
Apparently, the authors of these papers were not aware of the fact that for odd dimensions the conjecture has already been proved in \cite{Ch}, while for even dimensions a slightly weaker lower bound for the norm has been derived; see \cite[Corollary 3.4]{Ch}.  
The proof is based on the celebrated equivalence theorem due to Kiefer and Wolfowitz 
\cite{KW} that connects the problem with the approximate theory of optimal design.  For an extensive treatment of this theory we refer to \cite{Ki}.

In this note we give a short and transparent proof of the conjecture when $n$ is odd, while for even $n$ our method gives the same (weaker) lower bound as in \cite[Corollary 3.4]{Ch}. 

\begin{theorem}
Let $A \in \cD_n$ be a nonsingular matrix.

If $n \ge 3$ is an odd integer, then 
\begin{equation}
\label{SlHa}
\| A^{-1}\|_F \ge \frac{2 n}{n+1} \ , 
\end{equation}
where the equality holds if and only if  $A$ is an S-matrix.

If $n \ge 4$ is an even integer, then 
\begin{equation}
\label{even}
\| A^{-1}\|_F  > \frac{2 \sqrt{n^2 - 2 n + 2}}{n} \ , 
\end{equation}

If $n=2$ then 
\begin{equation}
\label{two}
\| A^{-1}\|_F \ge \sqrt{2} \ , 
\end{equation}
where the equality holds if and only if $A$ is either the identity matrix or 
$\left[ \begin{matrix} 
0 & 1 \cr
1 & 0  
\end{matrix}
\right]$.

\end{theorem}

\begin{proof} 
Let $e = (1, 1, 1,  \ldots, 1)^T \in \bR^n$ and $J = e e^T$ . We divide the proof into three cases.

CASE 1: $n \ge 3$ is an odd integer, so that $n = 2 k - 1$  for some  $k \in \bN$. \\
Define the matrices  $M$ and $N$ of order $n+1$ by 
$$  M = \left[
\begin{matrix} 
1 & e^T \cr
e & - k A^{-1} 
\end{matrix}
\right] 
\ \ \ \textrm{and} \ \ \
N = \left[
\begin{matrix} 
1 & e^T \cr
e & - (2I-\tfrac{1}{k} J) A^T 
\end{matrix} 
\right] . $$
Since 
$$  M N^T= \left[
\begin{matrix} 
n+1 &  *  \cr
*    &  (n+1) I 
\end{matrix}
\right]  , $$
we have 
$$
\langle M, N \rangle = \Tr ( M N^T) = (n+1)^2 . 
$$
By the Cauchy-Schwarz inequality, we then obtain 
\begin{equation}
\label{CS}
 (n+1)^4 =  (\langle M, N \rangle)^2 \le \| M \|_F^2 \cdot \| N \|_F^2 =  
\end{equation}
$$ = (1 + 2 n + k^2 \| A^{-1} \|_F^2) (1 + 2 n +   \| (2I-\tfrac{1}{k} J) A^T  \|_F^2)  \  .  $$
If we show that 
\begin{equation}
\label{bound}
\| (2I-\tfrac{1}{k} J) A^T  \|_F \le n \ , 
\end{equation}
then (\ref{CS}) gives the inequality
$$   (n+1)^2 \le 1 + 2 n + k^2 \| A^{-1} \|_F^2 , $$
and so 
$$  \| A^{-1} \|_F^2 \ge \frac{n^2}{k^2} = \left( \frac{2 n}{n+1} \right)^2 $$
completing the proof of (\ref{SlHa}).

To show  (\ref{bound}), we determine the maximum of the function $f$ defined on $\cD_n$ by
$$ f(A) = \| (2I-\tfrac{1}{k} J) A^T  \|_F^2 =  \Tr (A (2I-\tfrac{1}{k} J)^2 A^T) = \Tr \left( 4  A  A^T - \frac{2k +1}{k^2} (Ae)(Ae)^T \right) = $$
$$ = 4 \, \Tr(A  A^T) - \frac{2k +1}{k^2} (Ae)^T (Ae) = 
4 \sum_{i=1}^n \sum_{j=1}^n a_{i, j}^2  - \frac{2k +1}{k^2} \sum_{i=1}^n \left( \sum_{j=1}^n a_{i, j} \right)^2 , $$
where $A = [a_{i, j}]_{i, j =1}^n \in \cD_n$. 
Since $f$ is a continuous function on a compact set, it attains its maximum at some matrix $B = [b_{i, j}]_{i, j =1}^n \in \cD_n$.
Assume that $0 < b_{i, j} < 1$ for some $i, j \in \{1, 2, \ldots, n\}$. Then we have 
$$ \frac{\partial f}{\partial a_{i, j}} (B) = 0 \ \ \ \ \textrm{and} \ \ \ \ \   
   \frac{\partial^2 f}{\partial a_{i, j}^2} (B) \le 0 \ . $$
However, 
$$ \frac{\partial f}{\partial a_{i, j}} (A) = 8 a_{i, j} -  \frac{2k +1}{k^2} \,  2 \left( \sum_{l=1}^n a_{i, l} \right) , $$
and so 
$$ \frac{\partial^2 f}{\partial a_{i, j}^2} (A) = 8 -  \frac{2 (2k +1)}{k^2} = \frac{2 (4 k^2 - 2k - 1)}{k^2} > 0 $$
for all $A = [a_{i, j}]_{i, j =1}^n \in \cD_n$. Therefore, we conclude that $B$ is necessarily a $\{0, 1\}$-matrix.
Let $p_i$ be the number of ones in the $i$-th row of $B$. Then 
$$ f(B) = 4  \sum_{i=1}^n p_i -  \frac{2k +1}{k^2} \sum_{i=1}^n p_i^2 = 
- \frac{2k +1}{k^2} \sum_{i=1}^n \left(p_i -  \frac{2k^2}{2 k+1} \right)^2 + \frac{4 k^2}{2 k+1} (2 k - 1)  . $$
Since $k - \frac{2k^2}{2 k+1} = \frac{k}{2 k+1} \in (0, \frac{1}{2})$, we have $| m - \frac{2k^2}{2 k+1} | > | k - \frac{2k^2}{2 k+1} |$
for all $m \in \{1, 2, \ldots, n\} \setminus \{k\}$, implying that $p_i = k$ for all $i$.  It follows that
$$  f(B) = 4 n k - \frac{2k +1}{k^2} n k ^2 = n (4 k - 2 k - 1) =  n^2 . $$
This completes the proof of  the inequality (\ref{bound}).

Assume that the equality holds in  (\ref{SlHa}). Then there are equalities in (\ref{CS}) and (\ref{bound}), that is, $M = N$ and $A$ is an invertible  $\{0, 1\}$-matrix with $A e = k e$. 
It follows that $k A^{-1} = (2I-\frac{1}{k} J) A^T =  2 A^T -  J$, and so 
$e =  k A^{-1} e = 2 A^T e - (2 k - 1) e $  implying that  $A^T e = k e$. 
Therefore, we have 
$$  N N^T = M N^T = \left[
\begin{matrix} 
1 & e^T \cr
e & - k A^{-1} 
\end{matrix}
\right]  
 \left[ 
\begin{matrix} 
1 & e^T \cr
e &  J - 2 A 
\end{matrix} 
\right] = 
\left[
\begin{matrix} 
n+1 &  0  \cr
0   &  (n+1) I 
\end{matrix}
\right]  .$$
This means that 
$$ N^T = \left[ 
\begin{matrix} 
1 & e^T \cr
e &  J - 2 A 
\end{matrix} 
\right] $$
is a Hadamard matrix, and so $A$ is an $S$-matrix.
As the equality holds in  (\ref{SlHa}) when  $A$ is an $S$-matrix, the proof is complete for odd dimensions. 

CASE 2: $n \ge 4$ is an even integer, so that $n = 2k$ for some integer $k \geq 2$. \\
Define the matrices  $M$ and $N$ of order $n+1$ by 
$$  M = \left[
\begin{matrix} 
0 & e^T \cr
e & \tfrac{k \sqrt{k}}{\sqrt{k-1}} \, A^{-1} 
\end{matrix}
\right] 
\ \ \ \textrm{and} \ \ \
N = \left[
\begin{matrix} 
0 & e^T \cr
e &  \tfrac{\sqrt{k-1}}{k \sqrt{k}}  \left( \! \tfrac{k(2k-1)}{k-1} \, I - J \right) A^T   
\end{matrix} 
\right] . $$
Then 
$$  M N^T= \left[
\begin{matrix} 
2 k  &  *  \cr
*    &   \tfrac{k (2k-1)}{k-1} \, I 
\end{matrix}
\right] , $$
and so 
$$
\langle M, N \rangle = \Tr ( M N^T) = 2 k +  \frac{2 k^2 (2k-1)}{k-1} =  \frac{2 k (2k^2-1)}{k-1} . $$
By the Cauchy-Schwarz inequality, we then obtain 
\begin{equation}
\label{CSeven}
 \left( \frac{2 k (2k^2-1)}{k-1} \right)^2 = (\langle M, N \rangle)^2 \le \| M \|_F^2 \cdot \| N \|_F^2 =  
\end{equation}
$$ = (4 k + \tfrac{k^3}{k-1} \| A^{-1} \|_F^2) (4 k + g(A)) \ , $$
where $g(A)$ is defined by 
$$ g(A) = \left\| \tfrac{\sqrt{k-1}}{k \sqrt{k}}  \left( \! \tfrac{k(2k-1)}{k-1} \, I - J \right) A^T  \right\|_F^2 . $$
If  $A = [a_{i, j}]_{i, j =1}^n \in \cD_n$ then 
$$ g(A) = \Tr \left( A \left( \! \tfrac{(2k-1)^2}{k(k-1)} \, I - \tfrac{2}{k} \, J \right) A^T \right) =
 \frac{(2k-1)^2}{k(k-1)} \, \Tr (A A^T) - \frac{2}{k} (A e)^T \, (A e) = $$
$$ = \frac{(2k-1)^2}{k(k-1)} \sum_{i=1}^{2 k} \sum_{j=1}^{2 k} a_{i, j}^2  - \frac{2}{k} \sum_{i=1}^{2 k} \left( \sum_{j=1}^{2 k} a_{i, j} \right)^2 . $$
Since 
$$ \frac{\partial^2 g}{\partial a_{i, j}^2} (A) = \frac{2 (2k-1)^2}{k(k-1)} - \frac{4}{k} = \frac{2 (4 k^2 - 6 k + 3)}{k (k-1)} > 0 , $$
we conclude (similarly as in Case 1) that the maximum of the function $g$ on $\cD_n$ 
is attained at some matrix $B = [b_{i, j}]_{i, j =1}^n \in \cD_n$ with 
$b_{i, j} \in \{0, 1\}$ for all $i$ and $j$.   
Let $p_i$ be the number of ones in the $i$-th row of $B$. Then 
$$ g(B) = \frac{(2k-1)^2}{k(k-1)} \sum_{i=1}^{2 k} p_i -  \frac{2}{k} \sum_{i=1}^{2 k} p_i^2 = 
- \frac{2}{k} \sum_{i=1}^{2 k} \left(p_i -  \frac{(2k-1)^2}{4(k-1)} \right)^2 + \frac{(2k-1)^4}{4(k-1)^2} . $$
Since $k - \frac{(2k-1)^2}{4(k-1)} = - \frac{1}{4(k-1)} \in (-\frac{1}{2}, 0)$, 
we obtain that $p_i = k$ for all $i$.  It follows that
$$  g(B) = \frac{(2k-1)^2}{k-1} \, 2 k  - 4 k^2 =  \frac{2 k (2k^2-2k+1)}{k-1} . $$
Now, since $4 k + g(B) = \tfrac{2 k (2 k^2-1)}{k-1}$, 
the inequality (\ref{CSeven}) gives 
$$ 4 k + \tfrac{k^3}{k-1} \| A^{-1} \|_F^2 \ge \frac{2 k (2 k^2-1)}{k-1} , $$
and so 
\begin{equation}
\label{evenA}
\| A^{-1} \|_F^2 \ge \frac{2 (2k^2-2k+1)}{k^2} = \frac{4(n^2-2n+2)}{n^2} .
\end{equation}
To complete the proof of the inequality (\ref{even}), we must exclude the possibility of the equality in 
(\ref{evenA}). So, assume that for some matrix $A \in \cD_n$ the equality holds in (\ref{evenA}). Then $A$ is a $\{0, 1\}$-matrix and $M = N$. Therefore, we have 
$$ \frac{k^3}{k-1} \, A^{-1}  = \left( \! \frac{k(2k-1)}{k-1} \, I - J \right) A^T  $$
or 
$$ \frac{k^2}{2 k-1} I = \left( \! I - \frac{k-1}{k(2k-1)} \, J \right) A^T  A  , $$
implying that 
$$ A^T  A  = \frac{k^2}{2 k-1} \left( \! I - \frac{k-1}{k(2k-1)} \, J \right)^{-1} = 
\frac{k^2}{2 k-1} \left( \! I + \frac{k-1}{k} \, J \right) . $$
It follows that the off-diagonal entries of the matrix $A^T  A$ are equal to the number $\frac{k(k-1)}{2k-1}$
that is not an integer. This is a contradiction with the fact that $A$ is a $\{0, 1\}$-matrix.
 
CASE 3: $n =2$. 
If 
$$ A = \left[ \begin{matrix} 
a & b \cr
c & d  
\end{matrix}
\right] $$
is an invertible matrix in $\cD_2$, then 
$$ A^{-1}= \frac{1}{a d - bc} \, \left[ \begin{matrix} 
d & -b \cr
-c & a  
\end{matrix}
\right] , $$
and so 
$$ \| A^{-1}\|_F^2 =  \frac{a^2 + b^2 + c^2 + d^2}{(a d - b c)^2} . $$
Now, we have 
$$ (a d - b c)^2 \left( \| A^{-1}\|_F^2 - 2 \right) = (a-d)^2 + (b-c)^2 + 2 a d (1 - a d) + 2 b c (1 - b c) + 4 a b c d  \ge 0 . $$
We conclude that  
$$ \| A^{-1}\|_F^2 \ge  2 $$
and the equality holds if and only if $a=d$, $b=c$, $a d \in \{0, 1\}$, $b c  \in \{0, 1\}$ and $a b c d = 0$.
This implies the desired conclusions.
\end{proof}

\vspace{3mm}
{\bf
\begin{center}
 Acknowledgment.
\end{center}
} 
The author was supported in part by the Slovenian Research Agency.

\vspace{2mm}

\noindent
Roman Drnov\v sek \\
Department of Mathematics \\
Faculty of Mathematics and Physics \\
University of Ljubljana \\
Jadranska 19 \\
SI-1000 Ljubljana \\
Slovenia \\
e-mail : roman.drnovsek@fmf.uni-lj.si 


\begin{thebibliography}{999}

\bibitem{Ch}
C.-S. Cheng, 
{\it An application of the Kiefer-Wolfowitz equivalence theorem to a problem in Hadamard transform optics},
Ann. Statist. 15 (1987), no. 4, 1593--1603. 

\bibitem{HS}
M. Harwit and N.J.A. Sloane, Hadamard Transform Optics,  Academic, New York, 1979.

\bibitem{Hu} 
X. Hu, Some inequalities for unitarily invariant norms, 
J. Math. Inequal. 6 (2012), no. 4. 615--623.

\bibitem{Ki}
J. Kiefer, {\it General equivalence theory for optimum designs (approximate theory)}, Ann. Statist. 2 (1974), 849--879. 

\bibitem{KW}
J. Kiefer, J. Wolfowitz, {\it The equivalence of two extremum problems}, Canad. J. Math. 12 (1960), 363--366.

\bibitem{SH}
N.J.A. Sloane and M. Harwit, {\it Masks for Hadamard transform optics, and weighing designs}, 
Appl. Optics, 15 (1976), 107--114.

\bibitem{Zh}
X. Zhan, {\it Open problems in matrix theory}, 
in Proceedings of the 4th International Congress of Chinese Mathematicians, Vol. I, edited by L. Ji, K. Liu, L. Yang and S.-T. Yau, 
Higher Education Press, Beijing, 2008, 367--382.

\bibitem{Zou}
L. Zou,  On a conjecture concerning the Frobenius norm of matrices, Linear Multilinear Algebra 60 (2012), no. 1, 27--31. 

\bibitem{ZJH}
L. Zou, Y. Jiang, and X. Hu, 
A note on a conjecture on the Frobenius norm of matrices, 
J. Shandong Univ. (Nat. Sci.). 45 (2010), 48--50.

\end{thebibliography}
\end{document}